\documentclass[12pt]{amsart}
\usepackage[T1]{fontenc}
\usepackage{amsfonts}
\usepackage{mathrsfs}
\usepackage{amscd,amsmath,amssymb,amsfonts}
\usepackage{url}

\usepackage[a4paper,text={140mm,240mm},centering,headsep=5mm,footskip=10mm]{geometry}

\usepackage[all]{xy}
\theoremstyle{plain}
\newtheorem{thm}{Theorem}
\newtheorem{lem}[thm]{Lemma}
\newtheorem{cor}[thm]{Corollary}
\newtheorem{prop}[thm]{Proposition}

\theoremstyle{definition}
\newtheorem{defn}[thm]{Definition}

\newtheorem{question}[thm]{Question}
\newtheorem{rmk}[thm]{Remark}

\newtheorem{ex}[thm]{Example}

\newtheorem{claim}[thm]{Claim}

\numberwithin{thm}{section} \numberwithin{equation}{section}

\newcommand{\eq}[2]{\begin{equation}\label{#1}#2 \end{equation}}

\newcommand{\Hom}{{\rm Mor}}

\newcommand{\Spec}{{\rm Spec \,}}

\newcommand{\sO}{{\mathcal O}}

\newcommand{\C}{{\mathbf  C}}  
\newcommand{\D}{{\mathbf D}}    

\newcommand{\E}{{\mathbf E}}

\newcommand{\y}{{\bf y}}
\newcommand{\ay}{{\bf ay}}

\newcommand{\Z}{{\mathbb Z}}

\newcommand{\et}{{\rm \acute{e}t}}

\newcommand{\op}{{\rm op}}

\newcommand{\Sh}{\operatorname{Sh}}

\newcommand{\card}{\operatorname{card}}

\newcommand{\Mor}{{\rm Mor}}
\newcommand{\Pro}[1]{{{\rm pro}^{#1}\operatorname{-} }}
\newcommand{\colim}{\operatorname*{colim}}

\newcommand{\Mod}{\operatorname{Mod}}
\newcommand{\te}{{\rm t}}

\begin{document}

\title{Transfinite limits in topos theory}
\author{Moritz Kerz}
\address{Fakult\"at f\"ur Mathematik, Universit\"at Regensburg, 93040 Regensburg, Germany}
\email{moritz.kerz@mathematik.uni-regensburg.de}

\begin{abstract} 
For a coherent site we construct  a canonically associated enlarged coherent site, such
that cohomology of bounded below complexes is preserved by the enlargement. In the
topos associated to the enlarged site transfinite
compositions of epimorphisms are epimorphisms and a weak analog of the concept of the
algebraic closure exists.  The construction is a variant of the work
of Bhatt and Scholze on the
pro-\'etale topology.
\end{abstract}

\subjclass{18F10, 18F20}
\maketitle

\section{Introduction}

\noindent
In \cite{BS} B.\ Bhatt and P.\ Scholze construct a so called pro-\'etale enlargement of the usual \'etale topos
of a scheme, see also \cite[Tag 0965]{StP} and \cite{Sch}. A characteristic feature of the pro-\'etale topos is that certain limits have
better exactness properties than in the usual \'etale topos, while the cohomology of
classical \'etale sheaves does not change. This turns out to be useful when working with
unbounded derived categories. 

In this paper we propose a variant of the theory of Scholze and Bhatt--Scholze which works for a coherent site and we give two applications  to the
calculation of hypercohomology and to the existence of a left adjoint of the pullback of sheaves
along a closed immersion of schemes in the Nisnevich and \'etale topology.  

\medskip

Consider a coherent topos $\E$, as defined in \cite[Exp.\ VI]{SGA4}, for example the
\'etale topos of a quasi-compact and quasi-separated scheme. 
The key property we are interested in is whether in the topos a transfinite composition of
epimorphisms is an epimorphisms.  More precisely we say that $\E$ is $\alpha$-transfinite
if the following property holds:
 
For an ordinal $\lambda\le \alpha$ and for a functor $F:\lambda^{\rm op}\to \E$ with the
property that
\begin{itemize}
\item
 for any ordinal $1\le i+1<\lambda$ 
the morphism $F_{i+1}\to F_i$ is an epimorphism and
\item  for any limit ordinal
$\mu<\lambda$ the natural morphism
\[
F_\mu \xrightarrow{\sim} \lim_{i<\mu} F_i
\] 
is an isomorphism 
\end{itemize}
we ask that 
\[
 \lim_{i<\lambda} F_i\xrightarrow{\sim} F_0
\]
is an epimorphism.
Here the ordinal $\lambda$ as an ordered set is identified with the associated category.

The property $\aleph_0$-transfinite is studied in \cite{BS} under the name replete. The
topos of sets is $\alpha$-transfinite for all cardinals $\alpha$, while the
standard topoi that show up in algebraic geometry, for example the small \'etale topos, are usually not
$\aleph_0$-transfinite. So it is natural to try to make them transfinite in a minimal way.

In our first main theorem,  Theorem~\ref{thm1},  we construct for  any
coherent site $\C$ which is admissible in the sense of Definition~\ref{def.admiss} and for any infinite cardinal $\alpha$ a new coherent site $\langle \alpha
\rangle \C $ and a continuous functor preserving finite limits $\pi_\alpha^\C:\C\to \langle \alpha
\rangle \C$ such  that the topos $\langle\alpha \rangle \E  =\Sh (\langle \alpha
\rangle \C)$ is
$\alpha$-transfinite and the associated morphism of topoi
\[
( (\pi^\E_\alpha)_* , (\pi^\E_\alpha)^*):  \langle\alpha \rangle \E \to \E=\Sh(\C),
\]
has the property that $(\pi^\E_\alpha)^*$ is fully faithful and preserves cohomology of bounded below
complexes.  

In our second main theorem, Theorem~\ref{thm2}, we show that for large $\alpha$ the topos
$\langle\alpha \rangle \E$ 
is generated by weakly contractible objects. Here following \cite{BS} we call an object
$C$ of $\E$
weakly contractible if any epimorphism $D\to C$
in $\E$ splits. In some sense this means that the topoi $\langle\alpha \rangle \E$
`stabilize' for $\alpha$ large.
Note that in category theory it is more common to use the word projective
instead of weakly contractible. 

\smallskip

The main difference between our construction and the construction in \cite{BS} for the
 \'etale topos is that we work with a topology, which we call transfinite topology,
which sits between the usual \'etale topology and the pro-\'etale topology and in some sense
captures properties of both. The precise
relation is explained in Section~\ref{sec.zar} for the Zariski topos.

\medskip

Concretely our construction works as follows.
 We consider the pro-category $\Pro{\alpha} \C$ of
pro-objects whose index category is bounded by $\alpha$. We define in
Section~\ref{sec.pros} the transfinite topology on $\Pro{\alpha} \C$  as the weakest topology such that the
canonical functor
\[
\C\to \Pro{\alpha} \C
\]
is continuous and such that a transfinite composition of covering morphisms in $\Pro{\alpha} \C$
is a covering morphism. Then the site $\langle \alpha \rangle \C$ is just  $\Pro{\alpha}
\C$ with the transfinite topology.

\medskip

In order to motivate the construction of this paper we explain 
in Section~\ref{sec.hyper} why classical Cartan-Eilenberg hypercohomology of unbounded
complexes can be recovered as the derived cohomology on the enlarged topos $\langle\alpha
\rangle \E$.

Another motivation stems from the fact that, roughly speaking, in the world of transfinite enlarged topoi the
pullback functor of sheaves $i^*$ for a morphism of schemes $i:Y\to X$ tends to have a left
adjoint in the setting of Grothendieck's six functor formalism. This was observed for the
pro-\'etale topology in \cite[Rmk.\ 6.1.6]{BS} and the argument in our setting is very similar. Concretely, we show
that for $X$ quasi-compact and separated and for a closed immersion $i:Y\to X$ the
pullback functor
\[
i^*:D_\Lambda(\langle\alpha\rangle X_\te) \to D_\Lambda(\langle\alpha\rangle Y_\te)
\]
on derived categories of sheaves of $\Lambda$-modules has a left adjoint if $\alpha$ is large. Here $\te$ stands
for the small Nisnevich or \'etale topology on the category of affine, \'etale schemes over
$X$ or $Y$.

\subsection*{Notation} A category is called small if up to isomorphism its objects form a
set and not only a class. When we say topos we mean a Grothendieck topos. For topos theory
we follow the notation of \cite{MLM}. 

A coherent site is a small category having finite limits together with a topology generated
by finite coverings. For a subcanonical site $\C$ we write $\y:\C \to \Sh(\C)$ for the Yoneda embedding.

 A partially ordered set $(S,\le)$ is considered as a category with a
unique morphism $s_1\to s_2$ if $s_1\le s_2$ and no morphisms form $s_1$ to $s_2$
otherwise. 

By a $2$-category we mean a $(2,1)$-category, i.e.\  all $2$-morphisms are
invertible, $2$-functors between $2$-categories are allowed to be lax. So the
formalism of $\infty$-categories is applicable and we freely use notions from \cite{Lu}.

We use Zermelo--Fraenkel set theory including the axiom of choice. We do not use
the concept of universes as applied in \cite{SGA4}. 

\subsection*{Acknowledgment}

The results of this paper originated from discussions around a seminar on the work of B.\ Bhatt and
P.\ Scholze on the pro-\'etale topology \cite{BS} at the University of
Regensburg. I would like to thank all participants of this seminar. Clearly, this
paper is very much influenced by the work of Bhatt and Scholze.
I would like to thank B.\ Bhatt,  F.\ Strunk and M.\ Spitzweck for helpful discussions.
I would like to thank the referee for helpful remarks from the point of view of a category theorist.

The author was supported by the Emmy Noether group {\em Arithmetik \"uber endlich erzeugten
  K\"orpern} and the DFG grant SFB 1085. The paper was completed during a stay at the Institute for
Advanced Study.

\section{Preliminaries on towers and limits}\label{sec.tow}

\noindent
In this section we summarize some properties of pro-categories and diagrams indexed by
ordinals, which we call towers.

\subsection*{Pro-categories}
As a general reference for pro-categories see for example \cite{Isa}. In this paper we
need to bound the cardinalities of the index categories; however the basic arguments  essentially
stay the same as in the existing literature, so we do not give any proofs.

Let $\bf C$ and $ I$ be a categories and let $\alpha$ be an infinite cardinal. We call $ I$ an
$\alpha$-category if the system of all morphisms ${\rm Mor}({I})$ of $I$ forms a set of cardinality
at most $\alpha$. For a $\alpha$-category $ I$ and a functor $F:{ I} \to {\bf C}$ we
call $\lim_{i\in I} F(i)$ an $\alpha$-limit if it exists. There is a corresponding notion
for a functor to preserve $\alpha$-limits.

Note that the formation of $\alpha$-limits can be `decomposed'
into equalizers and products indexed by sets $I$ with $\card
I\le \alpha$ \cite[Sec.\ V.2]{ML}.

On can associate with $\bf C$ its pro-category $\Pro{\alpha}  \bf C$ indexed by
 cofiltered $\alpha$-categories.  
The objects of $\Pro{\alpha}  \bf C$ are the
functors 
\eq{}{
F:{I} \to \bf C
}
where $ I$ is cofiltered $\alpha$-category. For $F:{I} \to\bf C$ and $G:{ J}\to\bf C$ objects
of $\Pro{\alpha} \bf C$ the set of morphisms from $F$ to $G$ is given by
\eq{}{
\Mor_{\Pro{\alpha} \bf C}(F,G) = \lim_{j \in  J} \colim_{i\in I} \Mor_{\bf C} (F(i),G(j)).
}
Each object of $\Pro{\alpha} \bf C$ has a level representation $F: I \to\bf C$ with
$ I$ a cofinite directed set with $\card I\le \alpha$. This follows from the proof of
\cite[Prop.\ I.8.1.6]{SGA4}. All cofiltered $\alpha$-limits exist in $\Pro{\alpha} \bf C$, see
\cite[Thm.\ 4.1]{Isa}. If $\bf C$ has finite limits $\Pro{\alpha} \bf C$ has all
$\alpha$-limits.

Pro-categories can be characterized by the following universal property. Let ${\bf
  Cat}^{fl}$ be the $2$-category whose objects are small categories having finite limits,
whose $1$-morphisms are functors preserving finite limits and whose $2$-morphisms are all
natural equivalences. Let ${\bf Cat}^{l}$ be the $2$-category whose objects are all small
categories having $\alpha$-limits, whose $1$-morphisms are functors preserving
$\alpha$-limits and whose $2$-morphisms are natural equivalences.

\begin{prop}\label{prop.unipro}
The canonical $2$-functor ${\bf Cat}^{f} \to {\bf  Cat}^{fl}$ is right adjoint to the
$2$-functor mapping ${\bf C}\mapsto \Pro{\alpha}\bf C$. 
\end{prop}

For the notion of adjointness in higher category theory see for example \cite[Sec.\
5.2]{Lu}. Proposition~\ref{prop.unipro} is equivalent to the statement that there is a
natural equivalence of groupoids
\eq{}{
\Hom_{{\bf  Cat}^{fl}} ({\bf C}, {\bf D} ) \xrightarrow{\sim}_{\phi}\Hom_{{\bf  Cat}^{f}} ( \Pro{\alpha} {\bf C},{\bf D} )
}
for ${\bf C} \in {\bf  Cat}^{fl}$ and  ${\bf D} \in {\bf  Cat}^{l}$. This
equivalence is given as follows. For ${ F}\in \Hom_{{\bf  Cat}^{fl}} ({\bf C}, {\bf D} )$
and $(C_i)_{i\in  I}\in \Pro{\alpha} {\bf C}$ set \[\phi({ F}) ((C_i)_{i\in  I}) =
\lim_{i\in I} F(C_i).\]

\subsection*{Towers}

For an ordinal $\lambda$ consider a functor $ F:\lambda^\op
\to \bf C$. We usually denote such a functor by $ F=(F_i)_{i<\lambda}$.   For an
ordinal $\mu\le  \lambda$ write 
\eq{eq.limex}{
F_{<\mu} =
\lim_{i<\mu} F_i
}
if the limit exists.    We call $F$ a $\lambda$-tower (or just tower) if for any limit ordinal $\mu<
\lambda$ the limit \eqref{eq.limex} exists and if the natural morphism
\[F_\mu\to F_{<\mu}\]
is an isomorphism. 

 We say that the tower $F:\lambda^\op
\to \bf C$ has a certain property $\mathcal P$ if all the morphisms $F_{i+1}\to F_i$ have
the property $\mathcal P$ for $1\le i+1< \lambda$.
We call $F_{<\lambda}\to F_0$ the transfinite composition of the tower $(F_i)_{i<\lambda}$
if the limit exists.
 By a morphism of towers we mean a natural
transformation of functors. 

Let $\alpha$ be a cardinal and $\bf C$, $\bf D$ categories having all $\alpha$-limits. We
say that a functor $u:\bf D \to C$ preserves $\alpha$-transfinite limits if $u$ maps
$\lambda$-towers to $\lambda$-towers for $\lambda \le \alpha$. 

Let $ F=(F_i)_{i<\lambda}$ be a tower and $\pi_0:E_0\to F_{\mu}$ a
morphism for some  ordinal $\mu<\lambda$. If fiber products exist in $\bf C$ we define the pullback tower $ E=\pi_0^*
 F$ by \[ 
E_i= \begin{cases} E_0 & \text{if } i\le \mu \\ E_0 \times_{F_{\mu}} F_i  & \text{if } i> \mu \end{cases}
\]
There is a natural morphism of towers $\pi:E\to F$.

Let $F=(F_i)_{i<\lambda}$ and $G=(G_{j})_{j<\mu}$ be two towers. If $F_{<\lambda}\cong G_0$
 we consider the concatenation of towers $((F\circ G)_k)_{k< \lambda + \mu}$ with
 \[ 
(F\circ G)_k= \begin{cases} F_k & \text{if } k<\lambda \\ G_{j}  & \text{if } k=\lambda+j \end{cases}
\]
The concatention of two towers can be generalized to the concatenation of a 
family of towers indexed by an ordinal. We leave the details to the reader.

If we are given a symmetric monoidal structure $\odot: {\bf C} \times {\bf C}\to \bf C$
which preserves limits of towers and we are given two towers  $F=(F_i)_{i<\lambda}$ and
$G=(G_{j})_{j<\mu}$
we consider the tower $((F\odot G)_{k})_{k<\max(\lambda,\mu)}$. Without loss of
generality let $\lambda\le \mu$. Then,  assuming
$F_{<\lambda}$ exists, $F\odot G$ is defined by
 \[ 
(F\odot G)_k= \begin{cases} F_k \odot G_k & \text{if } k<\min(\lambda,\mu) \\
  F_{<\lambda}\odot G_k  & \text{if } k\ge \lambda \end{cases}
\]
For example we can use the categorical product for $\odot$ if it exists.

\section{Transfinite sites and topoi}

\noindent
In this section we study sites and topoi in which certain limits indexed by ordinal numbers behave well.
More precisely we call a topos transfinite if  transfinite compositions  of
epimorphisms are epimorphisms, in the sense of
towers as in Section~\ref{sec.tow}. The $\aleph_0$-transfinite topoi are the same as the
replete topoi of Bhatt and Scholze \cite[Sec.\ 3]{BS}.

Let $\alpha$ be an infinite  cardinal and let $\bf E$ be a topos. 

\begin{defn}\label{def.trtop}
We say that $\bf E$
is $\alpha$-transfinite if for any ordinal $\lambda\le \alpha$ and for any $\lambda$-tower
$(E_i)_{i<\lambda}$ of epimorphisms, i.e.\ with $E_{i+1}\to E_i$
an epimorphism for all $1\le i+1<\lambda$, the transfinite composition
\[
E_{<\lambda}=\lim_{i<\lambda} E_i\to E_0
\]
is an epimorphism. We say that $\bf E $ is transfinite if it is $\alpha$-transfinite for
all cardinals $\alpha$.
\end{defn}

\begin{ex}\label{ex.set.trans}
The topos  of sets $\bf Set$ is transfinite. For a group $G$ the topos $ {\bf B}G$ of $G$-sets
is transfinite.
\end{ex}

More generally, any topos with enough weakly contractible objects in the sense of
\cite[Def.\ 3.2.1]{BS}
is transfinite. 

\begin{defn}\label{def.weakcon}
We call an object $C$ of a topos $\bf
E$ weakly contractible, if any epimorphism $D\to C$ splits in $\bf E$, i.e.\ if there is a
morphism $C\to D$ such that the composition $C\to D\to C$ is the identity. We say that a topos
$\bf E$ has enough weakly contractible objects if for any object $C$ of $\bf E$ there is an
epimorphism $D\to C$ with $D$ weakly contractible.
\end{defn}

Note that a small coproduct of weakly contractible objects in a topos is weakly
contractible.

As any epimorphism splits in $\bf Set$, the topos of sets has enough weakly
contractible objects. The referee points out that the following proposition is a classical
fact about categories with enough projective objects: in fact in such a category a
morphism $f:C\to E$ is an epimorphism if and only if the induced map $\Hom (P,C)\to
\Hom(P,E)$ is surjective for any projective object $P$. Therefore
Proposition~\ref{prop.weakco} is reduced to Example~\ref{ex.set.trans}. As the latter
might not be well-known to a geometer and as we need a
variant of the proof in Corollary~\ref{prop.weaktrans}, we give a detailed argument below.

\begin{prop}\label{prop.weakco}
Let $\bf E$ be a topos with enough weakly contractible objects. Then $\bf E$ is transfinite.
\end{prop}

\begin{proof}
Let $ F=(F_i)_{i<\lambda}$ be a tower of epimorphisms in $\bf E$. 
Choose a weakly contractible $E_0$ and an epimorphism $\pi_0:E_0\to F_0$. Let $\pi:E\to F$
be the pullback tower along $\pi_0$. As the pullback of an epimorphism is an epimorphism
in a topos the tower $E$ consists of epimorphisms. 
In the commutative diagram
\[
\xymatrix{
F_{<\lambda}  \ar[d]_{\bf 1} &  \ar[l]  E_{<\lambda}  \ar[d]^{\bf 3} \\
F_0  &  \ar[l]_{\bf 2}  E_0 
}
\]
the morphisms $\bf 2$ and $\bf 3$ are epimorphisms by
Claim~\ref{cl.split}. So as $\bf 1$ is dominated by an epimorphism it is itself an epimorphism. 
\end{proof}

\begin{claim}\label{cl.split}
The morphism $E_{<\lambda}\xrightarrow{\bf 3} E_0$ splits. In particular it is an epimorphism. 
\end{claim}

\begin{proof}[Proof of claim]
We successively  construct a compatible family of splittings $(E_0\xrightarrow{s_i}
E_i)_{i<\mu}$ for $\mu\le \lambda$.
Compatible means that the diagram
\[
\xymatrix{
E_0 \ar[r]^{s_i}  \ar@/_1pc/[rr]_{s_j}& E_i \ar[r]  & E_j 
}
\]
commutes for all $j< i <\mu$. Assume the family of splittings has been constructed for
some $\mu< \lambda$. If
$\mu$ is a successor ordinal use the weak contractibility of $E_0$ to find $s_\mu$ such that the diagram
\[
\xymatrix{
   &  E_\mu \ar[d] \\
E_0  \ar[r]_{s_{\mu-1}}  \ar[ru]^{s_\mu} & E_{\mu-1}
}
\]
commutes. If $\mu$ is a limit ordinal let \[s_\mu=\lim_{i<\mu} s_i:E_0 \to  E_\mu \] be the
morphism obtained from the system $(s_i)_{i<\mu}$ by the universal property of the inverse limit and the isomorphism $E_\mu
\xrightarrow{\sim} \lim_{i<\mu}E_i$.

By this successive construction we can assume that there is a system of splittings $(E_0\xrightarrow{s_i}
E_i)_{i<\lambda}$. The morphism $\lim_{i<\lambda} s_i$ is a splitting of $\bf 3$.
\end{proof}

Another way, beside finding enough weakly contractible objects, to show that a topos is
transfinite, is to find a  site defining the topos in which transfinite
compositions  of
coverings are coverings. We will make this precise in the following.

\begin{defn}\label{def.admiss}
  A coherent site $\bf C$ is called admissible if its topology is subcanonical and for a
  finite family of objects $(C_i)_{i\in I}$ the coproduct $C=\coprod_{i\in I} C_i$ exists
  and $\{ C_i \to C \ |\ i\in I\}$ is a covering.  We furthermore assume that in $\C$ there is a
  strict initial object and coproducts
  are disjoint and
   stable under pullback, see \cite[App.]{MLM} and Definition~\ref{def.dctop}.
\end{defn}

\begin{lem}
The following are equivalent for a coherent subcanonical site $\C$:
\begin{itemize}
\item[(i)] $\C$ is admissible.
\item[(ii)] $\C$ has a strict inital object $\emptyset$ and the essential image of the
  Yoneda functor \[\y: \C \to \y(\emptyset)/ \Sh(\C)\] is closed under finite coproducts
  in the comma category $\y(\emptyset)/ \Sh(\C)$.
\end{itemize}
\end{lem}
 
Working with admissible sites instead of coherent sites is no real restriction as the following
lemma shows. For a  site $\C$ we denote by $\ay:\C \to \Sh(\C)$ the composition of the
Yoneda embedding and the sheafification.

\begin{lem}
For any coherent site $\C$ let $\overline{\C}$ be the smallest strictly full subcategory
of $\Sh(\C)$ which contains
the essential image of $\ay$ and which is closed under finite coproducts and finite
limits. Then $\overline{\C}$ with the epimorphic coverings is admissible and the continuous functor
$\ay:\C\to\overline{\C}$ induces an equivalence of topoi. 
\end{lem}


  Recall that a morphism $E\to D$ in $\bf C$ is
called a covering morphism if the sieve generated by $E\to D$ is a covering sieve.

\begin{defn}\label{def.trfin}
An admissible site $\bf C$ is called $\alpha$-transfinite if $\alpha$-limits exist in $\C$
and if
transfinite compositions of
$\lambda$-towers of covering morphisms ($\lambda\le \alpha$)  are
covering morphisms, i.e.\
we assume that for a $\lambda$-tower $(F_i)_{i<\lambda}$ in $\bf C$ with $F_{i+1}\to F_i$ a covering
morphism  for all
$i+1<\lambda$
 that $F_{<\lambda}\to F_0$ is a covering morphism.
\end{defn}

\begin{prop} \label{prop.trsite}
The topos $\Sh({\bf C})$ associated with an $\alpha$-transfinite site $\bf C$ is $\alpha$-transfinite. 
\end{prop}

\begin{proof}
Let $(F_i)_{i<\lambda}$ be a tower of epimorphisms in $\Sh({\bf C})$ ($\lambda \le
\alpha$). Choose a family $(C_r)_{r\in R}$ of objects in $\bf C$ and an epimorphism
$\pi_1:\coprod_{r\in R} \y(C_r) \to F_0$.
Recall that $\y:\C\to \Sh(\C )$ denotes the Yoneda embedding.
 For simplicity of notation we assume that
$R=\{ 0\}$ consists of only one element. Choose a family
$(C^{(1)}_r)_{r\in R_1}$ of elements of $\bf C$ and an epimorphism 
\eq{}{
\coprod_{r\in R_1} \y(C^{(1)}_r)\to  \y(C_0) \times_{F_0} F_1 .
}
As $\y(C_0)$ is quasi-compact there is a finite subset $\tilde R_1\subset R_1$ such that the
composite morphism
\eq{}{
\coprod_{r\in \tilde R_1} \y(C^{(1)}_r) \to \y(C_0) \times_{F_0} F_1 \to \y(C_0)
}
is an epimorphism. As the Yoneda functor is fully faithful, this morphism is induced by a covering morphism
$C_1=\coprod_{r\in \tilde R_1} C^{(1)}_r\to C_0$, see \cite[III.7 Cor.\ 7]{MLM}.
We get a commutative diagram
\[
\xymatrix{
\y(C_1) \ar[r] \ar[d] &  F_1 \ar[d]\\
 \y(C_0) \ar[r]& F_0
}
\]
which we are going to extend successively to the morphism of towers~\eqref{eq.motol}.

For doing so we assume now that  for $\mu<\lambda$ we have constructed a tower $(C_i)_{i< \mu}$  of covering morphisms  in $\bf C$ 
 and a morphism of towers $\pi_\mu:(\y(C_i))_{i<\mu}  \to F|_{\mu}$. If $\mu$ is a successor ordinal we proceed as
 above to find a covering morphism $C_\mu\to C_{\mu-1}$ and an extension  of $\pi_\mu$ to a morphism of
 towers 
\eq{eq.extt}{\pi_{\mu+1}:(\y(C_i))_{i\le \mu}  \to F|_{\mu+1}. }
If $\mu$ is a limit ordinal we let $C_\mu=\lim_{i<\mu}C_i$ and we let the morphism $
\y(C_\mu)= \lim_{i<\mu} \y(C_i)
\to F_\mu$ be the inverse limit of the morphism of towers $\pi_\mu$. This defines the
required extension as in \eqref{eq.extt} in the case of a limit ordinal $\mu$.

In the end this successive construction produces a tower of covering morphisms
$(C_i)_{i<\lambda}$ and a morphism of towers 
\eq{eq.motol}{
\pi:(\y(C_i))_{i<\lambda}  \to F.
}

The morphism $C_{<\lambda}=\lim_{i<\lambda} C_i\to C_0$ is the composition of a tower of covering morphisms, so
is a covering morphism itself, because $\bf C$ is
$\alpha$-transfinite.
In the commutative diagram
\[
\xymatrix{
\y(C_{<\lambda}) \ar[r] \ar[d]_{\bf 1}& F_{<\lambda} \ar[d]^{\bf 3}\\
\y(C_0) \ar[r]^{\bf 2}  & F_0  
}
\]
the morphism $\bf 1$ is the Yoneda image of a covering morphism and therefore an epimorphism. As also $\bf 2$ is an epimorphism, we see that $\bf 3$ is
dominated by an epimorphism and so is an epimorphism itself.
\end{proof}

\begin{ex}\label{ex.fpqc} For an infinite cardinal $\alpha$ let ${\bf Aff}_\alpha$ be the category of
  affine schemes $\Spec R$ with $\card (R)\le\alpha$. We endow ${\bf Aff}_\alpha$ with the
  fpqc-topology. Recall that the fpqc-topology on ${\bf Aff}_\alpha$ is generated by
  coverings $\{ U_i \to U\ | \ i\in I\}$ with $I$ finite, $U_i\to U$ flat and such that
  \[\coprod_{i\in I} U_i \to U\] is surjective.
Clearly, the site ${\bf Aff}^{\rm fpqc}_\alpha$ is $\alpha$-transfinite, so by
Proposition~\ref{prop.trsite} the fpqc-topos $\Sh({\bf Aff}^{\rm fpqc}_\alpha)$ is $\alpha$-transfinite.
\end{ex}

\section{Main theorems}

\noindent
Let $\alpha$ be an infinite cardinal.
Let $\bf Si$ be the $2$-category in the sense of \cite[XII.3]{ML} whose objects are
admissible sites $\bf C$ (Definition~\ref{def.admiss}), whose $1$-morphisms are
 continous functors $\C \to \D$ preserving finite limits  and whose $2$-morphisms are the
natural equivalences.
 Similarly, we consider the $2$-subcategory ${\bf Si}_\alpha$ of
 $\bf Si$ whose objects are the $\alpha$-transfinite sites  (Definition~\ref{def.trfin}) whose $1$-morphisms are the 
 continuous functors preserving  $\alpha$-limits and whose
 $2$-morphisms are all natural equivalences as above.

\begin{thm}\label{thm1}
For an infinite cardinal $\alpha$
the canonical functor of $2$-categories ${\bf Si}_\alpha \to {\bf Si}$ admits a left
adjoint \[\langle \alpha \rangle  : {\bf Si} \to {\bf
  Si}_\alpha .\]
For $\C$ admissible let $\E=\Sh(\C)$ and $ \langle \alpha \rangle \E =\Sh ( \langle \alpha
\rangle \C )$ be the associated topoi. The induced morphism of topoi $\pi_\alpha : \langle
\alpha \rangle \E \to \E $  has the property that $\pi_\alpha^*$ is fully faithful and
preserves cohomology of bounded below complexes of abelian sheaves.
\end{thm}

For the precise meaning of adjointness between $2$-categories in our sense see \cite[Sec.\
5.2]{Lu}.
The proof of Theorem~\ref{thm1} is given in the following two sections. In
Section~\ref{sec.pros} we define the site $ \langle \alpha \rangle \C$ as the category of
pro-objects $\Pro{\alpha} \bf C $ with the so called transfinite topology. In
Proposition~\ref{prop.proadm} we show that this site is admissible.  The adjointness
property is then obvious from the definition. The fact that $\langle \alpha \rangle$ is
fully faithful is immediate from Lemma~\ref{lem.pullpro}. The preservation of cohomology
is shown in Proposition~\ref{prop.compa}.

Unfotunately, we do not know whether the topos $ \langle \alpha \rangle \E$ depends on the site $C$ or
only on the topos $\E$.  
 Roughly speaking Theorem~\ref{thm1} means that for any admissible site $\C$
 we get a tower of topoi
\eq{eq.towtop}{
 \cdots \to\langle \aleph_{\lambda} \rangle {\bf E}  \xrightarrow{} \cdots  \to \langle \aleph_1 \rangle {\bf E}  \xrightarrow{}  \langle \aleph_0 \rangle {\bf E}  \xrightarrow{} {\bf E}
}
indexed by all ordinals $\lambda$, such that the higher up we get the topoi become `more transfinite'.
In fact our second main theorem tells us, see Corollary~\ref{cor.main}, that from some
point on
the topoi in the tower \eqref{eq.towtop} in fact are transfinite.

\begin{thm} \label{thm2}
For any admissible site $\bf C$ there is a cardinal $\beta$ such that for all cardinals
$\alpha\ge \beta$ the topos $\langle \alpha \rangle ({\bf E})=\Sh (\langle \alpha \rangle ({\C})) $ has enough weakly
contractible objects. More precisely, in $\langle \alpha \rangle ({\bf E})$ there exists a generating set of coherent, weakly contractible objects. 
\end{thm}

Recall that an object $C$ of $\E$ is quasi-compact if any covering family has a finite
subfamily which is covering. The object  $C$ is called coherent if it is quasi-compact and for any
quasi-compact objects $S,T$ of $\E$ and any two morphisms $S\to C$, $T\to C$ the object
$S\times_C T$ is quasi-compact \cite[Exp.\ VI.1]{SGA4}. 

The proof of Theorem~\ref{thm2} is given in the first part of Section~\ref{sec.weak}. 
Using Proposition~\ref{prop.weakco} we deduce:

\begin{cor}\label{cor.main}
For any admissible site $\bf C$ there is a cardinal  $\beta$ such that for all cardinals
$\alpha\ge \beta$ the topos $\langle \alpha \rangle {\bf E}$ is transfinite.
\end{cor}

\begin{rmk}\label{rmk.weakco}
The cardinal $\beta$ in Theorem~\ref{thm2} and Corollary~\ref{cor.main} can be chosen to
be $\card (\Mor(\C))$. More precisely
$\beta$ can be chosen in such a way that the admissible site
$\C$ is $\beta$-small. For the notion of smallness see Definition~\ref{def.small}. 
\end{rmk}

\section{The pro-site of a coherent site}
\label{sec.pros}

\noindent
Let $\bf C$ be a coherent site and let $\alpha$ be an infinite cardinal.
We are going to construct two topologies on the pro-category $\Pro{\alpha} \bf C$ defined in Section~\ref{sec.tow}, such that the embedding of categories $\bf
C \to  \Pro{\alpha} \bf C$ is continuous, i.e.\ maps coverings to coverings. Recall that
this embedding also preserves finite limits.

\subsection*{Weak topology}
The weak topology on $ \Pro{\alpha} \bf C$ is defined as the weakest topology such that
the functor  $\bf
C \to  \Pro{\alpha} \bf C$ is continuous.

Clearly, for any covering morphism $V\to W $ in $\C$ and for a morphism $U\to W $ in  $
\Pro{\alpha} \C$ the base change $V\times_W U \to U$ is a covering morphism in the weak
topology. We call such weak covering morphisms distinguished. Similarly, if $\{ W_i \to W\
| \ i\in I \}$ is a finite covering in $\C$ the family $\{ W_i\times_W U\to U\ |\ i\in I
\}$ is a weak covering in  $\Pro{\alpha} \C$, which we call distinguished.

One can give an explicit level representation of the distinguished weak coverings. Let $F:{I} \to \bf C$ be
an object of $\Pro{\alpha} \bf C$. We assume that $I$ has a final element $i_\circ$ and that there is
given a covering $\{ C_w\to F(i_\circ) \ |\ w\in W\}$ in the site $\bf C$. Let $F_w:{I}
\to \bf C$ be the functor given by $F_w(i)=F(i) \times_{F(i_\circ )} C_w$.  Then
\eq{eq.procov}{
\{F_w\to F\ |\ w\in W  \}  
} 
is a distinguished covering in  $\Pro{\alpha} \bf C$ and all distinguished coverings are
of this form up to isomorphism.

\begin{prop} \label{prop.weaktop} For a coherent site $\C$
the weak topology on $\Pro{\alpha} \bf
C$ is coherent and has as a basis the coverings which have level representations of the
form \eqref{eq.procov}, i.e.\ the distinguished weak coverings.
\end{prop}

\begin{proof}
We have to show that the system of distinguished weak coverings defines a basis $\mathcal B$ for a topology on $\Pro{\alpha} \bf
C$. Clearly, an isomorphism is
a covering in $\mathcal B$ and the pullback of a covering in $\mathcal B$ exists and is
itself a covering in $\mathcal B$ by definition.

The  property we have to check is that the composition of coverings from
$\mathcal B$ is a
covering in $\mathcal B$. More precisely, let $\{F_w\to F\ |\ w\in W  \}$ be a covering in
$\mathcal B$ of the form \eqref{eq.procov}, i.e.\ with a level representation indexed by
the cofiltered $\alpha$-category 
$ I$ with final element $i_\circ$. Given coverings  $\{G_{w,v}\to F_w\ | \ v\in W_w
\}$ in $\mathcal B$ for $w\in W$ we have to show that the composite morphisms 
\eq{eq.covcomp}{\{
G_{w,v}\to F\ | \ w\in W, v\in W_w \}
}
form a covering in $\mathcal B$. Changing the level representation (here we use that $W$ is finite) we can assume that the
$G_{w,v}$ are also indexed by $ I$ and that $G_{w,v}(i)=F_w(i)\times_{F_w(i_\circ)}
D_{w,v}$ for all $i\in I$. Here $\{D_{w,v}\to F_w(i_\circ)\ | \ v\in W_w \}$ are coverings in $\bf C$. So
\eqref{eq.covcomp} is level equivalent to the pullback of the covering $\{D_{w,v}\to
F(i_\circ)\ | \ w\in W, v\in W_w \} $ along $F\to F(i_\circ)$ and therefore is a covering in
 $\mathcal B$.
\end{proof}

\begin{defn}\label{def.small} Let $\alpha$ be an infinite cardinal.
We say that a site $\bf C$ is $\alpha$-small if for any object $C$ in $\bf C$
there is a set of covering morphisms $K(C)$ of $C$ with $\card K(C)\le \alpha$ such that
for any covering morphism
$E \to C$ in $\bf C$ there is $D\to C$ in $K(C)$ and a factorization $D\to E\to C$.
\end{defn}

Clearly, any coherent site whose underlying category is an $\alpha$-category is $\alpha$-small.

\begin{prop}\label{prop.small}
If the coherent site $\bf C$ is $\alpha$-small the pro-site $\Pro{\alpha}\bf C$ with the
weak topology is also $\alpha$-small.
\end{prop}

\begin{proof}
Consider $F:{ I} \to \bf C$ in  $\Pro{\alpha}\bf C$ with $I$ a directed set with $\card I
\le \alpha$. For every $i\in  I$ let $K_i$ be
a set of covering morphisms of $F(i)$ in $\bf C$ as in Definition~\ref{def.small}. By
Proposition~\ref{prop.weaktop} the set of cardinality at most $\alpha$ of covering morphisms 
\[
 \{ D \times_{F(i)} F\to F \ | \  i\in I, (D\to F(i))\in K_i \}
\]
satisfies the condition of Definition~\ref{def.small} for the weak topology.
\end{proof}

In the next lemma we collect for later reference a few fact about coproducts in $\Pro{\alpha}(\C)$. 

\begin{lem}\label{lem.coprodpro} Assume $\C$ is an admissible site, see Definition~\ref{def.admiss}.
\begin{itemize}
\item[(i)] A strict initial object in $\C$ defines a strict initial object in
  $\Pro{\alpha}(\C)$.
\item[(ii)] $\Pro{\alpha}(\C)$ has finite coproducts which are disjoint and stable under
  pullback. Furthermore, finite coproducts of towers are towers. 
\item[(iii)] For a finite coproduct $U=\coprod_{i\in I} U_i$  in $\Pro{\alpha}(\C)$ the
  family 
$
\{ U_i \to U\ |\ i\in I\}
$
is a distinguished weak covering.
\item[(iv)] For a distinguished weak covering $\{U_i \to U \ |\ i\in I\}$ in
  $\Pro{\alpha}(\C)$  the morphism $\coprod_{i\in I} U_i \to U$ is a distinguished weak
  covering morphism.
\item[(v)] For a finite family of distinguished weak covering morphisms $V_i\to U_i$ in
  $\Pro{\alpha}(\C)$ ($i\in I$) the morphism $\coprod V_i\to \coprod U_i$ is a
  distinguished weak covering morphism.
\end{itemize}
\end{lem}

\begin{proof}\mbox{}\\
{\em (ii): }
Use that finite coproducts commute with cofiltered $\alpha$-limits in $\Pro{\alpha}(\C)$ by
\cite[Thm.\ 6.1]{Isa}.\\
{\em (iii):}
Choose common level representations $(U_i(j))_{j\in J}$ of the $U_i$ ($i\in I$) such that
$J$ has the final element $j_\circ$. We know
that $(\coprod_i U_i(j))_{j\in J}$ is a level representation for $U$, which we fix.  As $U_i(j) \xrightarrow{\sim}
U_i(j_\circ )\times_{U(j_\circ)} U(j) $ is an isomorphism (use the strict initial object), we see that  $
\{ U_i \to U\ |\ i\in I\}
$ 
is the pullback of the covering $\{ U_i(j_\circ ) \to U(j_\circ ) \ | \ i\in I\}$ in $\C$.
\end{proof}

\subsection*{Transfinite topology}\label{sec.transtop}

The transfinite topology on $\Pro{\alpha} \bf C$ is the weakest topology such that the functor
${\bf C} \to \Pro{\alpha} \bf C$ is continuous and such that $\lambda$-transfinite
compositions of covering morphisms are covering morphisms ($\lambda\le \alpha$). The latter means that if
$(F_i)_{i<\lambda}$ is a tower in $\Pro{\alpha} \bf C$ with $ \lambda\le \alpha$ such that
$F_{i+1}\to F_i$ is a covering morphism for all $i+1<\lambda$ the morphism 
\[
F_{<\lambda}=\lim_{i<\lambda}F_i \to F_0
\]
is a covering morphism. 

The category $\Pro{\alpha} \bf C$ with the transfinite topology is
denoted $\langle \alpha \rangle {\bf C}$. In Proposition~\ref{prop.proadm} we show that $\langle \alpha \rangle {\bf C}$
is admissible if $\C$ is admissible. 
This will complete
the proof of the adjointness part of Theorem~\ref{thm1} in view of Proposition~\ref{prop.unipro}.

A key step is to give an explicit presentation of the transfinite topology for an
admissible site
 $\bf C$, see Definition~\ref{def.admiss}.
For this consider transfinite coverings in $\Pro{\alpha} \bf C$ of the following form. We call
a morphism $\tilde U \to U$ in $\Pro{\alpha} \bf C$ a distinguished transfinite covering morphism if it
is an $\lambda$-transfinite composition ($\lambda\le\alpha$) of distinguished weak covering morphisms. The families of the form
\eq{eq.stcov}{
\{ U_w\to \tilde U \to U \  | \ w\in W \}
} 
with $\tilde U\to U$ a distinguished transfinite covering morphism and $\{U_w \to \tilde U\ | \ w\in W\}$
a distinguished weak covering ($W$ finite) are transfinite coverings, called distinguished
transfinite coverings.

\begin{prop}\label{prop.sttop}
If $\bf C$ is an admissible site the transfinite topology on $\Pro{\alpha} \bf C$ is coherent
and has as a basis the distinguished transfinite coverings, i.e.\ the coverings of the form \eqref{eq.stcov}. 
\end{prop}

\begin{proof}
First we show that the coverings \eqref{eq.stcov} form a basis $\mathcal B$ for a
topology.
The only nontrivial part is to check that the composition of coverings in $\mathcal B$ is in
$\mathcal B$.

Let
\[
\{ U_w\to \tilde U \to U \  | \ w\in W \}
\]
be in $\mathcal B$ and for all $w\in W$ let
\[
\{ U_{w,v}\to \tilde U_w \to U_w \  | \ v\in W_w \}
\]
be in $\mathcal B$. 

The morphism $\coprod_{w\in W} U_w \to \tilde U$ is a distinguished weak covering morphism
by Lemma~\ref{lem.coprodpro}(iv).
In  $\Pro{\alpha}(\C)$ finite coproducts  of towers  are towers and finite coproducts of distinguished weak covering morphisms are
distinguished weak covering morphisms by Lemma~\ref{lem.coprodpro}(ii) and (v).  So by concatenation of towers we get that the composition
\[
\coprod_{w\in W} \tilde U_w \to \coprod_{w\in W} U_w \to \tilde U \to U
\]
is a distinguished transfinite covering morphism. As \[\{ U_{w,v} \to \coprod_{w'\in W} \tilde U_{w'} \ | \
w\in W, v\in W_w  \}\]
is a distinguished weak covering we have shown that 
\[
\{ U_{w,v} \to U\ |\ w\in W, v\in W_w   \}
\]
is in $\mathcal B$.

In order to finish the proof of Proposition~\ref{prop.sttop} we have to show that
$\lambda$-transfinite compositions of covering morphisms with respect to the topology
defined by $\mathcal B$ are covering morphisms in the same topology ($\lambda\le\alpha$). By an argument very
similar to the proof of Proposition~\ref{prop.trsite} one is 
 reduced to showing that for a tower $(U_i)_{i<\lambda}$ of distinguished
transfinite covering morphisms the composition $U_{<\lambda}\to U_0$ is a distinguished
transfinite covering morphism. By assumption for any $i+1<\lambda$ we can find a tower
$(U_{i,j})_{j< \lambda_i}$ of distinguished weak covering morphisms such that $U_{i,0}=U_i$
and $\lim_{j<\lambda_i} U_{i,j}=U_{i+1}$. By transfinite concatenation of the towers $(U_{i,j})_{j\in \lambda_i}$  we get a
tower $(U'_k)_{k<\lambda'}  $ of distinguished weak covering morphisms indexed by the ordinal $\lambda'= \sum_{i<\lambda}\lambda_i $ 
with $U'_0=U_0$ and $U'_{<\lambda'}=U_{<\lambda}$. So $U_{<\lambda}\to U_0$ is a
distinguished transfinite covering morphism.
\end{proof}

We say that an object $U$ of a site is weakly contractible if any
covering morphism $V\to U$ of the site splits. Clearly, if the site is subcanonical this
is equivalent to saying that the sheaf $\y (U)$ is weakly contractible in the associated topos in the sense of Definition~\ref{def.weakcon}.

\begin{cor}\label{prop.weaktrans}
Let $\C$ be an admissible site and let $U \in \Pro{\alpha} \C$ be weakly contractible for
the weak topology. Then $U$ is also weakly contractible for the transfinite topology.
\end{cor}

The proof of Corollary~\ref{prop.weaktrans} is very similar to the proof of
Proposition~\ref{prop.weakco}, so we omit the details. It is sufficent to show that a
distinguished transfinite covering morphism $\lim_{i<\lambda} U_i \to U$ splits. Here
$U_{i+1}\to U_i$ are distinguised weak covering morphisms for all $i$. Such a splitting can
be constructed successively over $i$ as splittings $U\to U_i$.

\section{Pro-covering morphisms}

\noindent
Let $\bf C$ be an admissible site and $\alpha$ an infinite cardinal.  In this section we
collect a few results which are related to the concept of pro-covering morphism.

\begin{defn}
A morphism $f:V\to U$ in  $\Pro{\alpha} \bf C$ is a pro-covering morphism if
$f$ has a level representation by covering morphisms in $\bf C$.
\end{defn}

\begin{lem}\label{lem.procov}
A distinguished covering morphism in the weak and in the transfinite topology is a pro-covering morphism.
\end{lem}

\begin{proof}
The case of the weak topology is trivial by the description \eqref{eq.procov} of
distinguished weak covering morphisms.

Let $(U_i)_{i<\lambda}$ ($\lambda\le \alpha$) be a tower of distinguished weak covering
morphisms. We want to show that $U_{<\lambda}\to U_0$ is a pro-covering morphism. Without
loss of generality $\lambda $ is not a limit ordinal. We argue by contradiction. If the
composition is not a pro-covering morphism there exists a smallest ordinal $\mu<\lambda$
such that $U_{\mu}\to U_0$ is not a pro-covering morphism. 

If $\mu$ is a successor ordinal
$U_\mu\to U_{\mu-1}\to U_0$ is a composition of a distinguished weak covering morphism and a
pro-covering morphism, so it is a pro-covering morphism, contradiction. 

If $\mu$ is a limit ordinal \[U_\mu\xrightarrow{\sim} \lim_{i<\mu} U_i \to U_0  \]
is a cofiltered limit of pro-covering morphisms, so is a pro-covering morphism by
\cite[Cor.\ 5.2]{Isa}, contradiction.
\end{proof}

\begin{prop}\label{prop.proadm}
For $\bf C$ admissible the site $\Pro{\alpha} {\bf C }$ with the weak and the transfinite
topology is admissible.
\end{prop}

By what is shown in Section~\ref{sec.transtop} the site $\Pro{\alpha} {\bf C }$ with both topologies is
coherent. 
The site $\Pro{\alpha} {\bf C }$ is subcanonical by Lemmas~\ref{lem.procov} and~\ref{lem.procovcan}.
It has finite coproducts with the requested properties by Lemma~\ref{lem.coprodpro}.

\begin{lem}\label{lem.procovcan}
For a pro-covering morphism $f:V\to U$  and for an object $W$ in $\Pro{\alpha} {\bf C }$
\[
\Hom(U,W) \to \Hom(V,W) \rightrightarrows \Hom(V\times_U V ,W)
\]
is an equalizer. 
\end{lem}

\begin{proof}
Without loss of generality we can assume $W\in \bf C$.
Let $(V_i \xrightarrow{ f_i}  U_i)_{i\in I}$ be a level representation of $f$ with $f_i$ a
covering morphism in $\bf C$. For each $i\in I$ we get an equalizer
\eq{eq.eqfori}{
\Hom(U_i,W) \to \Hom(V_i,W) \rightrightarrows \Hom(V_i\times_{U_i} V_i ,W).
}
Taking the colimit over $i\in I$ in \eqref{eq.eqfori} and using the fact that in the
category of sets filtered colimits commute with finite limits \cite[IX.2]{ML} we finish
the proof of Lemma~\ref{lem.procovcan}.
\end{proof}

Let $\pi:\C \to \Pro{\alpha}(\C)$ be the canonical functor.

\begin{lem}\label{lem.pullpro}
For a sheaf $K$ on $\bf C$ the sheaf $\pi^* K$ on $\Pro{\alpha} {\bf C }$ with the weak or
transfinite topology is given on $U=(U_i)_{i\in I}$ by
\[
\pi^* K (U) = \colim_{i\in I} K(U_i).
\]
\end{lem}

\begin{proof}
For any sheaf $L$ on ${\bf C }$ consider the presheaf
\[
L^\sharp:U=(U_i)_{i\in I}\mapsto \colim_{i\in I} L(U_i) 
\] 
 on $\Pro{\alpha} {\bf C }$.
For a pro-covering morphism $W\to U$ 
\[
L^\sharp (U) \to L^\sharp (W) \to L^\sharp (W\times_U W)
\]
is an equalizer, because in $\bf Set$ finite limits commute with filtered colimits. So by Lemma~\ref{lem.procov} it follows that $L^\sharp$ is a sheaf. 

By \cite[Prop.\ I.5.1]{SGA4} the presheaf pullback of $K$ to a
presheaf on  $\Pro{\alpha} {\bf C }$ is given by 
\eq{eq.fupullpro}{
U\mapsto \colim_{(V,f_V)\in U/\pi} K(V)
}
where $U/\pi$ is the comma category whose objects consist of $V\in \C$ and a
morphism $U\to \pi (V)$ in  $\Pro{\alpha} {\bf C }$. As the objects $(U_i,U\to U_i)$ are
cofinal in this comma category we see that the presheaf \eqref{eq.fupullpro} coincides
with the sheaf $K^\sharp$, which finishes the proof by~\cite[Prop.\ III.1.3]{SGA4}.
\end{proof}

\medskip

We conclude this section with an application of the notion of pro-covering morphism to
derived categories.
For a commutative unital ring $\Lambda$ let $D^+_\Lambda( {\bf C})$ be the derived
category  of bounded below complexes of $\Lambda$-modules. The following proposition is a
variant of \cite[Cor.\ 5.1.6]{BS}.

\begin{prop}\label{prop.compa} For both the weak and the transfinite topology and
for $K\in D^+_\Lambda( {\bf C})$ the natural transformation \[K \to {\bf R}\pi_* \pi^*
K  \]
is an equivalence.
\end{prop}

\begin{proof}
One easily reduces to the case of an injective sheaf $K $ in $\Sh_\Lambda ({\bf C})$.
As $K\xrightarrow{\sim} \pi_* \pi^* K$ is an isomorphism by Lemma~\ref{lem.pullpro}, we have to show that \[H^j(U_{\Pro{\alpha}\bf
  C},\pi^* K)=0 \quad \text{ for all } j>0 \text{ and } U \text{ in } \Pro{\alpha}\bf C \]
in the weak and in the transfinite topology, because this implies that 
$\pi_* \pi^* K \to
{\bf R} \pi_*\pi^* K$ is a quasi-isomorphism by \cite[Prop.\ V.5.1]{SGA4}.
By~\cite[Prop.\ V.4.3]{SGA4} it suffices to show that \v{C}ech cohomology 
\[
\check{H}^j( U_{\Pro{\alpha}\bf C},K) = \colim_{f\in {\bf
    Cov}_{ \Pro{\alpha}\bf C}(U)} \check{H}^j(f,K) 
\]
vanishes for $j>0$. Here ${\bf Cov}_{ \Pro{\alpha}\bf C}(U)$ is the category of distinguished covering
morphisms of $U$ in the weak resp.\ transfinite topology.  For simplicity of notation we do not distinguish between $K$ and
$\pi^* K$. As the distinguished covering morphisms are pro-covering morphisms by Lemma~\ref{lem.procov}, 
 $f\in {\bf Cov}_{ \Pro{\alpha}\bf C}(U)$ has a level
  representation of the form $(V_i \xrightarrow{f_i} U_i)_{i\in I}$ with covering
  morphisms $f_i$ in $\bf C$.
Again by~\cite[Prop.\ V.4.3]{SGA4} and using injectivity of $K$ as a sheaf on $\bf C$ we obtain the vanishing of 
\[
 \check{H}^j(f,K)  = \colim_{i\in I}  \check{H}^j(f_i,K) = 0 \quad \text{ for } j>0.
\]
\end{proof}

\section{Weakly contractible objects}\label{sec.weak}

\subsection*{Proof of existence}

\noindent
In this subsection we prove Theorem~\ref{thm2}. Consider the topos $\E=\Sh (\C )$, where $\C$
is an admissible site. Let $\beta$ be an infinite cardinal such that $\C$ is
$\beta$-small, see Definition~\ref{def.small}. The site $\Pro{\alpha}\C$ with the weak
topology is $\alpha$-small for $\alpha\ge\beta$ by Proposition~\ref{prop.small}. We are
going to show that under this condition for
any object $U$ in $\Pro{\alpha}\C$ there is a transfinite covering morphism $\mathcal P^\infty(U)\to U$ such
that $\mathcal P^\infty(U)$ is weakly contractible in the weak topology. Then by
Corollary~\ref{prop.weaktrans} $\mathcal P^\infty(U)$ is also weakly contractible in the transfinite
topology and this clearly implies that $\langle \alpha\rangle \E =\Sh(\langle \alpha
\rangle \C)$ has a generating set of coherent weakly contractible objects.

So consider $\alpha\ge  \beta$. Choose for each $U$ in $\Pro{\alpha}\C$ a set of
cardinality at most $\alpha$ of generating
 covering morphisms $K(U)$ as in Definition~\ref{def.small}.
Let $\mathcal P (U)$ be the product $\prod_{(V\to U) \in K(U)} (V\to U)$ in the comma category
$\Pro{\alpha}\C / U$.

\begin{claim}
For each $U$ in $\Pro{\alpha}\C$
 the morphism $\mathcal P (U) \to U$ is a transfinite covering morphism in $\Pro{\alpha}\C$.
\end{claim}
\begin{proof}
Let $\lambda\le \alpha$ be an ordinal such that there is a bijection $$\iota:\{ i<\lambda \ |\ i \text{ is successor
  ordinal }\} \xrightarrow{\sim} K(U).$$
We successively construct a tower of weak covering morphisms $(V_i)_{i<\lambda}$ with
$V_0=U$ and $V_{<\lambda}=\mathcal P (U)$. Assume $V_j$ has already been defined for all
$j<i$. If $i<\lambda$ is a successor ordinal set \[V_i=V_{i-1} \times_U \iota(i).\] If
$i<\lambda$ is a limit ordinal set \[ V_i = \lim_{j<i} V_j. \]
\end{proof}

For a positive integer $i$ let $\mathcal P^i (U)$ be the $i$-fold application of $\mathcal
P$, i.e.\ \[\mathcal P^i(U) = \mathcal P ( \mathcal P^{i-1}(U) ) \quad\text{ for } i>1,\]
and let $\mathcal P^\infty(U)$ be $\lim_{i\in \mathbb N} \mathcal P^i(U)$. By
concatenation of towers we see that $\mathcal P^\infty(U)\to U$ is a $\lambda$-transfinite
compositon of weak covering morphisms ($\lambda\le \alpha$).

\begin{claim} The object
$\mathcal P^\infty(U)$ of $\Pro{\alpha}\C$ is weakly contractible in the weak topology.
\end{claim}

\begin{proof}
Let $V'\to  \mathcal P^\infty(U)$ be a distinguished weak covering morphism. There
exists a positive integer $i$ and a distinguished weak covering morphism $V\to \mathcal P^i(U)$ such
that \[V'\cong V\times_{\mathcal P^i(U)} \mathcal P^\infty (U).\]
By the definition of $\mathcal P$ there is a factorization \[\mathcal P^{i+1}(U) \to V  \to
\mathcal P^i(U)\] of the canonical morphism $\mathcal P^{i+1}(U) \to \mathcal P^{i}(U)$,
which induces a splitting of $V'\to \mathcal P^{\infty}(U)$.
\end{proof}

\subsection*{Disjoint covering topology}

\begin{defn}\label{def.dctop}
We call a small category $\D$ a dc-category if finite coproducts exist in $\D$ and
furthermore finite coproducts are disjoint and stable under pullback, see \cite[App.]{MLM}.
The finite coverings of the form $\{ V_i \to V\ |\ i\in I\}$ with $V=\coprod_{i\in I} V_i$ define a basis for a topology
on $\D$,
which we call the dc-topology.
\end{defn}

\begin{lem}\label{prop.dcpull} Let $\C$ be an admissible site.
\begin{itemize}
\item[(i)] The full subcategory of weakly contractible objects $\tilde \C$ in $\C$ forms a
  dc-category and the functor $\tilde \C \to \C$ is continuous, see \cite[Def.\ III.1.1]{SGA4}.
\item[(ii)]
If there are enough weakly contractible objects in $\C$, i.e.\ if for any object $U$ in
$\C$ there is a covering morphism $V\to U$ with $V\in \C $ weakly contractible
the restriction of
 sheaves induces an equivalence of categories between $\Sh(  \C) $ and $\Sh (\tilde \C)$.
 Here $\tilde \C$ has the dc-topology.
\end{itemize}
\end{lem}

To show Lemma~\ref{prop.dcpull}(ii) one uses the comparison lemma \cite[App., Cor. 4.3]{MLM}.

For a ring $\Lambda$ and a topos $\E$ let $\Mod_\Lambda(\E)$ be the category  of
$\Lambda$-modules in $\E$.

\begin{lem}\label{lem.exprop}
Let $\E$ be a topos and let $U$ be a weakly contractible object in $\E$.
The additive functor  from $\Mod_\Lambda(\E)$ to $\Lambda$-modules
\[
F \mapsto \Gamma(U,F)=\Mor_\E (U,F)
\]
is  exact.
\end{lem}

\section{Example: Cartan--Eilenberg hypercohomology}\label{sec.hyper}

\noindent 
Let $C$ and $D$ be admissible sites.
Let $f: {\bf D}\to {\bf C}$ be a continuous functor preserving finite limits. For a commutative unital
ring $\Lambda$ let $\Mod_\Lambda({\bf C})$ be the category of sheaves of $\Lambda$-modules
on $\bf C$ and let $D_\Lambda({\bf C})$ be its derived category.

In geometry one is often interested in studying the right derived functor ${\bf R}f_* :
D_\Lambda({\bf C})\to D_\Lambda({\bf D})$. It was shown by Joyal and Spaltenstein \cite{Sp} that this right derived functor
always exists abstractly, see for example \cite{CD} for a modern account. However, it has
good `geometric' properties only for complexes bounded below or under some condition of
finite cohomological dimension. These problematic aspects of the right derived functor are
discussed in the framework of homotopy theory in \cite[Sec.\ 6.5.4]{Lu}. 

As an alternative to the derived functor one can use the
older concept of Cartan--Eilenberg hypercohomology  pushforward 
 \[\mathbb H f_*:
D_\Lambda({\bf C})\to D_\Lambda({\bf D}) \] 
defined for a complex $K^*$ as 
$\mathbb H f_* (K^*)= f_*( {\rm Tot} I^{* * }) $, where $K^* \to I^{* *}$ is a
Cartan--Eilenberg injective resolution \cite[Sec.\ XVII.1]{CE} and where $({\rm Tot} I^{* * })_n= \prod_{i+j=n}
I^{i, j}$. In this form Cartan--Eilenberg hypercohomology is studied in \cite[App.]{We}. In fact, in \cite{CE} the direct sum instead of the direct product
is used, but this does not seem to be appropriate for cohomology. 
Cartan--Eilenberg
hypercohomology is equivalent to hypercohomology calculated using the Godement resolution, see \cite[App.]{WG}. 

For admissible sites we can give a universal characterization of Cartan--Eilenberg
hypercohomology in terms of  derived functors. Let \[\pi_\alpha^{\bf C}: \C\to  \langle
\alpha \rangle {\bf C} \]
be the canonical functor. We denote the induced functor $   \langle
\alpha \rangle {\bf D}\to  \langle
\alpha \rangle {\bf C} $ by $f^\alpha$. 

\begin{prop}\label{prop.carei} For coherent sites $\bf C$ and $\bf D$
and an infinite cardinal $\alpha$ the diagram
\[
\xymatrix{
D_\Lambda( \langle \alpha \rangle {\bf C} ) \ar[r]^-{{\bf R} f^\alpha_*} &  D_\Lambda( \langle
\alpha \rangle {\bf D} ) \ar[d]^{{\bf R}(\pi_\alpha^{{\bf D}})_*  }   \\
 D_\Lambda( {\bf C }) \ar[u]^{(\pi_\alpha^{{\bf C}})^*  }  \ar[r]_{\mathbb H f_*} &  D_\Lambda( \bf D)
}
\]
commutes up to canonical equivalence.
\end{prop}

\begin{proof}
It is sufficient to show that for a complex $K$ of sheaves of $\Lambda$-modules on $\bf C$   there is
a quasi-isomorphism 
\[
{\bf R}(\pi_\alpha^{{\bf C}})_*    (\pi_\alpha^{{\bf C}})^*  K \simeq {\rm Tot} I^{* * }  ,
\]
where $K^* \to I^{* * }$ is a Cartan--Eilenberg injective resolution as above. 

\begin{claim}
The functor $(\pi_\alpha^{{\bf C}})_*$ maps a Cartan--Eilenberg injective resolution of $
(\pi_\alpha^{\bf C})^* K$ on $\langle
\alpha \rangle {\bf C}$ to a  Cartan--Eilenberg injective resolutions of $K$.
\end{claim}

\begin{proof}[Proof of claim]
Note that $(\pi_\alpha^{{\bf C}})_*$ preserves injective sheaves and products because it
is a right adjoint of the exact functor $(\pi_\alpha^{{\bf C}})^*$.
One easily reduces the proof of the claim to the case in which $K$ is in $\Mod_\Lambda(\C)$. Let $(\pi_\alpha^{\bf C})^* K\to I^*$ be an injective resolution 
of $K$. Then by Proposition~\ref{prop.compa}  the pushforward $(\pi_\alpha^{\bf
  C})_* I $ is an injective resolution of $K\simeq (\pi_\alpha^{\bf C})_*
(\pi_\alpha^{\bf C})^* K$, so the claim follows. 
\end{proof}

 Using the claim Proposition~\ref{prop.carei} follows immediately from \cite[Thm.\ A.3]{We}.
 Here we use that countable products are exact in $\Mod_\Lambda(\langle \alpha \rangle {\bf C}) $, see
 \cite[Prop.\ 3.1.9]{BS}, which is sufficient in the proof of \cite[Thm.\ A.3]{We}.
\end{proof}

\section{Example: transfinite Zariski topos}\label{sec.zar}

\noindent
In this section we explain how the construction of Section~\ref{sec.pros}  applied to the Zariski topos of an affine
scheme $X=\Spec (R)$ relates to the method of Bhatt--Scholze \cite[Sec.\ 2]{BS}. The
comparison in the \'etale case is very similar.
We fix an infinite cardinal $\alpha$ with $\alpha \ge \card (R)$

\subsection*{The category $ {\bf Aff}_X$}
Let $ {\bf Aff}_X$ be the category of affine schemes $\Spec (A)$ over $X$ with $\card(A)\le\alpha$. The Zariski topology on $
{\bf Aff}_X$ has a basis given by coverings
\[
\{ \Spec(A[\frac{1}{f_i}])\to \Spec (A)\ |\ i=1,\ldots ,n \}
\]
where $f_1,\ldots , f_n\in A$ generate the unit ideal in $A$.

\subsection*{ The site  $ {\bf Aff}^{\rm Zar}_X$}
  Let  $ {\bf Aff}^{\rm Zar}_X$
be the full subcategory of $ {\bf Aff}_X$ whose objects are of the form $\Spec \prod_{i=1}^n R[1/f_i]$ with
$f_1,\ldots ,f_n\in R$ 
and whose morphisms are scheme morphisms over $X$. We endow $ {\bf Aff}^{\rm Zar}_X$ with
the Zariski topology.
In \cite[Sec. 2.2]{BS} the objects of $ {\bf Aff}^{\rm Zar}_X$ are called Zariski localizations of $R$.

Clearly, the associated topos $\Sh( {\bf Aff}^{\rm Zar}_X)$ is equivalent to the usual
Zariski topos of $X$. Moreover,   $ {\bf Aff}^{\rm Zar}_X$ is admissible, see
Definition~\ref{def.admiss}.

\subsection*{The category  ${\bf Aff}^{\rm pro}_X$}  
The functor 
\eq{eq.lim}{
\lim : \Pro{\alpha}{\bf Aff}^{\rm Zar}_X \to {\bf Aff}_X
}
which maps a pro-system to its inverse limit is fully faithful. For any affine scheme
$\Spec A$ in the image of the functor \eqref{eq.lim} Bhatt--Scholze say that $A$ is an ind-Zariski
localizations of $\Spec R$. We write the 
 image of the functor \eqref{eq.lim} as $ {\bf Aff}_X^{\rm pro}$.

\subsection*{Topologies on   ${\bf Aff}^{\rm pro}_X $  }

The topology on ${\bf Aff}^{\rm pro}_X$ induced by the Zariski topology on ${\bf Aff}_X$ is isomorphic to the weak topology on
$\Pro{\alpha}{\bf Aff}^{\rm Zar}_X$ via the equivalence induced by \eqref{eq.lim}.

In \cite{BS} and \cite[Tag 0965]{StP} the pro-\'etale topology is studied. There is an obvious analog in
the Zariski word, the pro-Zariski topology, defined as follows:
\[
\{ U_i \xrightarrow{\pi_i} U\ | \ i\in I \}
\]
is a pro-Zariski covering if $I$ is finite, $\coprod_{i\in I} U_i \to U$ is surjective and
$\pi_i$ induces an isomorphism $\sO_{U,\pi_i(x)}\to \sO_{U_i,x}$ for all $ x\in U_i$.

We get the following relations between topologies on ${\bf Aff}^{\rm pro}_X $
\[
\text{(Zariski topology)} \subset \text{(transfinite topology)} \subset \text{(pro-Zariski topology)}.
\]

\begin{question}
Does there exist an analog of the pro-Zariski topology on $\Pro{\alpha}\C$ for a general
admissible site $\C$. This pro-topology should be  stronger than the transfinite topology. For example one might
try to define the requested pro-topology as generated by coverings  $\{ U_w \to
U\ | \ w\in W\}$ with $W$ finite and with $\coprod_{w\in W} U_w \to U$ a pro-covering
morphism which induces a surjection on topos points.
\end{question}

\subsection*{The category  ${\bf Aff}^{\rm oil}_X $  }

One problem of the pro-category ${\bf Aff}^{\rm pro}_X $ is that its definition is not
local on $X$. This is the reason why in \cite{BS} and \cite[Tag 0965]{StP} the weakly
\'etale morphisms and in the Zariski case the isomorphisms of local rings morphisms are
used. A similar technique, which is related to the pro-\'etale topology of rigid spaces as defined in \cite{Sch}, can be used in our case in order to replace ${\bf Aff}^{\rm
  pro}_X$ by a more local definition.

Consider the full subcategory of ${\bf Aff}_X$ consisting  of universally open morphisms 
$f:Y\to X$ which identify local rings, i.e.\  for any point $y\in Y$ the map
$f^*:\sO_{X,f(y)}\xrightarrow{\sim} \sO_{Y,y}$ is an isomorphism.

\begin{lem}\label{lem.prosur}
For a functor $F:I\to {\bf Aff}^{\rm oil}_X$ such that $I$ is a cofiltered
$\alpha$-category and such that all transition maps $F(i)\to F(j)$ are surjective the
limit $Y=\lim_{i\in I} F(i)$ taken in ${\bf Aff}_X$ is an object of ${\bf Aff}^{\rm oil}_X$.
\end{lem}

\begin{proof} We show that $Y\to X$ is open.
Any affine open subscheme $U\subset Y$ is the preimage of some affine open $U_i\subset
F(i)$ for some $i\in I$. Note that $Y\to F(i)$ is surjective, because the fibres of the
transition maps in the system $F$ are finite and nonempty and a cofiltered limits of
finite nonempty sets is nonempty. So the image of $U$ in $X$ is the same as the image of
$U_i$ and therefore is open.
\end{proof} 

If not mentioned otherwise we endow  ${\bf Aff}^{\rm oil}_X$ with the transfinite
topology, i.e.\ the weakest topology containing the Zariski coverings and such that a
$\lambda$-transfinite composition of  covering morphisms is a covering morphism ($\lambda\le\alpha$).
This topology has an explicit description similar to  Proposition~\ref{prop.sttop}. 

For $U$ in ${\bf Aff}^{\rm oil}_X$ the weakly
contractible object $\mathcal P^\infty (U)$ as defined in Section~\ref{sec.weak} exists in
${\bf Aff}^{\rm oil}_X$. So the site  ${\bf Aff}^{\rm oil}_X$ has  similar properties
as  ${\bf Aff}^{\rm pro}_X$ with the transfinite topology. In fact both are closely
related as we show now.

\begin{prop}
For any $Y\to X$ in  ${\bf Aff}^{\rm oil}_X$ there is a $\lambda$-transfinite composition ($\lambda\le\alpha$)
of surjective Zariski localizations $\tilde Y\to Y$ such that $\tilde Y\to X$ is in  ${\bf
  Aff}^{\rm pro}_X$.

In particular there is an equivalence of topoi \[\Sh({\bf Aff}^{\rm oil}_X)\cong \Sh(
{\bf Aff}^{\rm oil}_X\cap  {\bf Aff}^{\rm pro}_X  ), \]
where both sites have the transfinite topology.
\end{prop}

\begin{proof}[Sketch of Proof]
Composing $Y\to X$ with the transfinite composition of surjective Zariski localizations
$\mathcal P^\infty (Y) \to Y$ we can assume without loss of generality that $Y$ is weakly
contractible.

Consider the following data: $Y=\coprod_{i\in I} V_i$ is a finite decomposition into open
and closed affine subschemes and $U_i\subset X$ is an open affine subscheme such that
$f(V_i)\subset U_i$.
The set of such data forms a directed set $\bf J$ under the ordering by refinement.
Then
\[
Y\xrightarrow{\sim} \lim_{\bf J} \coprod_{i\in I} U_i
\]
is an isomorphism.
\end{proof}

\section{Example: a left adjoint to $i^*$}\label{sec.lead}

\noindent
Let $i:Y\to X$ be a closed immersion of separated, quasi-compact schemes. Consider the
category of schemes $U$ together with an \'etale, affine morphism $U\to X$.  We write
$X_\te$ for this category endowed either with the Nisnevich  ($\te={\rm Nis}$) or \'etale
($\te=\et$) topology, similarly for $Y$. Clearly, $X_\te$ and $Y_\te$ are admissible.

 Let $\Lambda$ be a commutative ring. Consider the pullback functor
\eq{pullo}{
i^*: D_\Lambda(Y_\te ) \to D_\Lambda(X_\te)
}
on unbounded derived categories of complexes of $\Lambda$-modules.

Our aim in this section is to show that the analogous pullback functor in the transfinite Nisnevich and transfinite
\'etale topology  has a left adjoint. A similar result for the pro-\'etale topology
has been observed in~\cite[Rmk.\ 6.1.6]{BS}. Before discussing the transfinite case we
discuss why in the classical case the functor \eqref{pullo} has no left adjoint in general.

Recall that the derived categories in \eqref{pullo} have small products. For $(K_i)_{i\in I}$ a
family of complexes  $\Lambda$-modules in $X_\te$ the infinite product of these
complexes in $D_\Lambda(X_\te)$ is calculated by first replacing the $K_i$ by K-injective
complexes as in \cite{Sp}, see e.g. \cite{CD}, and then taking degreewise products of sheaves.
The following example shows that in general the functor \eqref{pullo} does not preserve
infinite products, in particular it cannot have a left adjoint.

\begin{ex}
For a prime $p$ write the henselization of $\Z_{(p)}$ as an filtered direct limit of
\'etale $\Z_{(p)}$-algebras  
\[
\Z_{(p)}^h = \colim_{j\in J} A_j.
\]
We consider the closed immersion 
\[
i:\Spec(\mathbb F_p ) \to \Spec (\Z_{(p)}).
\]
By $\Z[A_j]$ we denote the \'etale sheaf of free abelian groups on $X_\et=\Spec
(\Z_{(p)})_\et$ represented by $\Spec (A_j)$. Then the homotopy limit in the sense of
\cite[Sec.\ 1.6]{Ne} taken in the triangulated category $D_\Z (X_\et )$
\[
K=\operatorname{holim}_{j\in J} \Z[A_j]
\]
has vanishing \'etale cohomology sheaf in degree zero. However $i^* \Z[A_j]$ is the
constant sheaf $\Z$, so
\[
\operatorname{holim}_{j\in J} i^* \Z[A_j] = \Z
\]
is not quasi-isomorphic to $i^* K$.
\end{ex}

\smallskip

\begin{thm}\label{thm.pullbacki}
 Let $X$ be quasi-compact and separated.
There exists an infinite cardinal $\beta$ such that for $\alpha \ge \beta$ the functors
\begin{align}
\label{eq.pull1} i^*:  \Sh(\langle \alpha\rangle X_\te ) &\to  \Sh( \langle \alpha\rangle Y_\te)  \\
\label{eq.pull2} i^*: D_\Lambda ( \langle \alpha \rangle X_\te ) &\to D_\Lambda ( \langle \alpha \rangle
Y_\te ) 
\end{align}
have left adjoints.
\end{thm}

\begin{proof}
Choose $\beta$ such that for any open affine subscheme $\Spec(A) \to X$ we have $\beta\ge
\card(A)$. Then according to Theorem~\ref{thm2} and Remark~\ref{rmk.weakco} there exists a generating
set of coherent, weakly
contractible objects in the  topoi  $\Sh(\langle \alpha \rangle X_\te)$  and $\Sh(\langle \alpha
\rangle Y_\te )$.

\begin{lem}\label{lem.cogen}
Any coherent topos $\E$ as a small cogenerating set.
\end{lem}

\begin{proof}
By Deligne's theorem~\cite[Sec.\ IX.11]{SGA4} any coherent topos $\E$ has a set of points 
\[
(p_{j,*},p_j^*):{\bf Set} \to \E \quad\quad (j\in J)
\]
such that all $p_j$ together induce a faithful functor \[(p_j^*)_{j\in J}:\E \to \prod_{j\in
  J} {\bf Set} . \]
The set of objects $p_{j,*}( \{1,2\})$ ($j\in J$) is cogenerating.
\end{proof}

\begin{lem}\label{lem.comgen}
The triangulated categories $D_\Lambda(\langle \alpha\rangle X_\te)$ and $D_\Lambda(\langle
\alpha\rangle Y_\te)$ are compactly generated. 
\end{lem}

\begin{proof}
For simplicity of notation we restrict to $D_\Lambda(\langle \alpha\rangle X_\te)$.
For $U$  a coherent, weakly contractible object in $\Sh(\langle \alpha\rangle X_\te)$ the sheaf of free $\Lambda$-modules
$\Lambda[U]$ represented by $U$ is a compact object of the triangulated category
$D_\Lambda ( \langle \alpha \rangle X_\te )$. In fact the global section functor
$\Gamma(U,-)$ preserves exact complexes by Lemma~\ref{lem.exprop}. Furthermore, taking sections over a
coherent object preserves direct sums of $\Lambda$-modules \cite[Thm.\ VI.1.23]{SGA4}.

Let $\mathcal W$ be a set of such coherent, weakly contractible objects $U$ 
  which generate the topos
$\Sh(\langle \alpha\rangle X_\te)$. Then the set of compact objects $\{\Lambda[U]\ |\ U\in \mathcal W\}$ generates the
triangulated category $D_\Lambda ( \langle \alpha \rangle X_\te )$.
\end{proof}

By the special adjoint functor theorem \cite[Sec.\ V.8]{ML} the existence of a left adjoint 
to \eqref{eq.pull1} follows once we show that the functor \eqref{eq.pull1} preserves small products. Indeed,
coherent topoi satisfy the conditions of the special adjoint functor theorem by
Lemma~\ref{lem.cogen} and general properties of topoi.

 By \cite[Prop.\ 5.3.1]{Kr} and by Lemma~\ref{lem.comgen}
the existence of the left adjoint to \eqref{eq.pull2} follows if we show that \eqref{eq.pull2}
preserves small products.

\smallskip

In order to prove that our two functors $i^*$ preserve small products we can assume without loss of generality that
$X$ is affine. In this case the fact that \eqref{eq.pull1} preserves products is immediate
from Lemma~\ref{cl.ladsh}. The argument for the functor \eqref{eq.pull2} is given after the proof of
Lemma~\ref{cl.ladsh}.

\begin{lem}\label{cl.ladsh}
For $\alpha$ as above and for $X$ affine the functor
\[
i^*:\Sh (\langle\alpha\rangle X_\te ) \to \Sh (\langle\alpha\rangle Y_\te)
\]
has a left adjoint, denoted $i_\flat$, which satisfies
\begin{itemize}
\item[(i)] $ i^*\circ i_\flat\simeq {\rm id}$,
\item[(ii)] $i_\flat$ maps weakly contractible objects to weakly contractible objects. 
\end{itemize}
\end{lem}

\begin{proof}[Proof of Lemma~\ref{cl.ladsh}]
For $V\to
Y$ affine \'etale there exists an affine \'etale scheme $U'\to X$ such that $U'\times_X
Y\cong V$, see \cite[Tag 04D1]{StP}. Let $U$ be the
henselization of $U'$ along $V$, see \cite[Ch.\ XI]{Ra}. The resulting affine scheme $U$ together with the isomorphism
$U\times_X Y\cong V$ is unique up to unique isomorphism and depends functorially on $V$.

Taking inverse limits defines a fully faithful functor from $\langle\alpha\rangle X_\te$ to the
category of affine schemes over $X$. And the scheme $U$ constructed above lies in the
essential image of this functor. Without loss of generality we will identify $U$ with an object of
$\langle\alpha\rangle X_\te$. 

So the map $V\mapsto U$ extends to
a functor $i_\flat^{\rm pre}:Y_\te \to \langle\alpha\rangle X_\te$ which we can extend by continuity to a
functor
\eq{eq.iflatpre}{
i_\flat^{\rm pre}:\langle\alpha\rangle Y_\te \to \langle\alpha\rangle X_\te.
}
which is left adjoint to the pullback functor $ U \mapsto U \times_X Y $.

By \cite[Prop.\ I.5.1]{SGA4} the pullback along $i$ in the sense of presheaves  maps a presheaf $F$ on $\langle\alpha\rangle X_\te$ to the presheaf
\eq{}{
V\mapsto \colim_{(U,f_U)} F(U),
}
on $\langle\alpha\rangle Y_\te$,
where $(U,f_U)$ runs through the comma category of all pairs in which $U$ is in $\langle\alpha\rangle X_\te$ and $f_U$
is a map $V\to U\times_X Y$ in $\langle\alpha\rangle Y_\te$. But clearly for given $V$ the object
$i_\flat^{\rm pre} V$ in $\langle\alpha\rangle X_\te$ together with the isomorphism $i_\flat^{\rm
  pre} V\times_X Y\cong  V$ is an initial element in the comma category of these pairs. So
the presheaf pullback of $F$ is given by
\eq{eq.lefaduni}{
V\mapsto F( i_\flat^{\rm  pre} V ).
}

Let $\tilde Y_\te$ be the full subcategory of $\langle\alpha\rangle Y_\te$ given by the weakly
contractible objects. Note that according to the first part of Section~\ref{sec.weak}, the objects
of $\tilde Y_\te$ generate $\langle \alpha \rangle Y_\te$.  The restriction of the functor~\eqref{eq.iflatpre} to $\tilde
Y_\te$ with the dc-topology, see
Definition~\ref{def.dctop}, is continuous in the sense of \cite[Def.\ III.1.1]{SGA4}. To see
this, note that decompositions into disjoint unions of $V\in \tilde Y_\te $ can be lifted to $i^{\rm
  pre}_\flat V$ by \cite[Prop.\ XI.2.1]{Ra}. Now \cite[Prop.\ III.1.2 ]{SGA4} tells that that we
get an adjoint pair of functors
\[
i_\flat^{\rm res}:\Sh (\tilde Y_\te) \rightleftharpoons \Sh (\langle\alpha\rangle X_\te ): i^*_{\rm res} .
\]
By Lemma~\ref{prop.dcpull} and by what is explained above equation
\eqref{eq.lefaduni}, the right adjoint $i^*_{\rm res}$ is just
the composition of the sheaf pullback 
\eq{eq.shpullb}{
i^*:\Sh( \langle\alpha\rangle X_\te)\to\Sh( \langle\alpha\rangle Y_\te)
}
composed with the restriction to $\tilde Y_\te$. However the latter restriction
\[
\Sh(\langle\alpha\rangle Y_\te)\xrightarrow{\sim} \Sh( \tilde Y_\te )
\]
 is an equivalence of categories according to Lemma~\ref{prop.dcpull}, so we obtain
 a left adjoint $i_\flat$ to \eqref{eq.shpullb} as stated in Lemma~\ref{cl.ladsh}.

 Property (i) of the Lemma is immediate from the construction of $i_\flat$ and (ii)
 follows abstractly from the adjointness property.
\end{proof}

Consider a family of complexes of sheaves of $\Lambda$-modules $(K_j)_{j\in J}$ on
$\langle\alpha\rangle X_\te$. Note that, because there are enough weakly contractible objects in
$\Sh(\langle \alpha \rangle X_\te)$, small products of complexes of sheaves of
$\Lambda$-modules on $\langle \alpha \rangle X_\te$ preserve quasi-isomorphisms 
by Lemma~\ref{lem.exprop}.  So we have to show that
\[
i^* (\prod_{j\in J} K_j)\to \prod_{j\in J} i^* K_j 
\]
is a quasi-isomorphism, where the product is just the degreewise product of sheaves. Using
 compact generators $\Lambda[U]$  of $D_\Lambda ( \langle
\alpha \rangle Y_\te ) $, see Lemma~\ref{lem.comgen}, it suffices to show that 
\eq{}{
\Gamma(U, i^* (\prod_{j\in J} K_j)) \to\prod_{j\in J} \Gamma(U, i^* K_j )
}
is a quasi-isomorphism of complexes of $\Lambda$-modules. By
the adjunction of Lemma~\ref{cl.ladsh} this is equivalent to showing that 
\[
\Gamma(i_\flat U, \prod_{j\in J} K_j) \to \prod_{j\in J} \Gamma(i_\flat U,   K_j)
\]
is a quasi-isomorphism, which is obvious. 
\end{proof}

\bibliographystyle{plain}

\renewcommand\refname{References}

\end{document}